\documentclass[11pt]{amsart}

\usepackage[off]{auto-pst-pdf} 
\usepackage{pstricks,pst-node,pst-tree}

\newtheorem{thm}[equation]{Theorem}
\newtheorem{cor}[equation]{Corollary}
\newtheorem{lem}[equation]{Lemma}
\newtheorem{prop}[equation]{Proposition}

\theoremstyle{definition}

\numberwithin{equation}{section}

\newcommand{\depth}{\mathsf{d}}
\newcommand{\rootdeg}{\mathsf{rdeg}}

\newcommand{\m}{\mathsf{m}}
\newcommand{\rt}{\mathsf{rt}}

\newcommand{\cc}{\mathbb{C}}
\newcommand{\z}{\mathbb{Z}}

\newcommand{\n}{\mathfrak{n}}
\newcommand{\h}{\mathfrak{h}}

\newcommand{\g}{\mathfrak{g}}

\newcommand{\W}{\mathcal{W}}

\newcommand{\T}{\mathbb{T}}

\newcommand{\C}{{\bf c} \,}
\newcommand{\lb}{\lbrace \hspace{-.03in}\mid  \hspace{-.045in} }
\newcommand{\rb}{ \hspace{-.045in} \mid \hspace{-.03in} \rbrace}

\definecolor{mjo}{rgb}{.4,0,.9}

\newcommand{\Hom}{\mbox{Hom}}

\begin{document}

\title[Whittaker modules for the insertion-elimination Lie algebra]
{Whittaker modules for the insertion-elimination Lie algebra}

\author[Matthew Ondrus and Emilie Wiesner]{Matthew Ondrus and Emilie Wiesner}

\address{\noindent Mathematics Department,
Weber State University, 
Department of Mathematics, 
1415 Edvalson St., Dept. 2517, 
Ogden, UT 84408-2517 USA,  \emph{E-mail address}: \tt{mattondrus@weber.edu}}

\address{\noindent Department of Mathematics, Ithaca College, Williams Hall Ithaca, NY 14850, USA, \ \  \emph{E-mail address}: \tt{ewiesner@ithaca.edu}}

 
\thanks{\textbf{MSC Numbers (2010)}: Primary: 17B10, 17B65;  \ Secondary: 17B70 \hfill \newline
}

\date{}

\maketitle

\begin{abstract}
This paper addresses the representation theory of the insertion-elimination Lie algebra, a Lie algebra that can be naturally realized in terms of tree-inserting and tree-eliminating operations on rooted trees.  The insertion-elimination algebra admits a triangular decomposition in the sense of Moody and Pianzola, and thus it is natural to define a Whittaker module corresponding to a given algebra homomorphism. Among other results, we show that the standard Whittaker modules are simple given certain constraints on the corresponding algebra homomorphism. 
\end{abstract}

\section{Introduction}

The notion of an insertion-elimination algebra was introduced by Connes and Kreimer \cite{CK} as a way of describing the combinatorics of inserting and collapsing subgraphs of Feynman graphs.  Connes and Kreimer investigated Hopf algebras related to rooted trees; the insertion-elimination algebra arises in relation to the dual algebra of one of these Hopf algebras.  Further results focusing on the Hopf algebra perspective have been obtained by Hoffman \cite{Hoff} and Foissy \cite{Fos}.  The operations of insertion and elimination can be described in terms of rooted trees, and Sczesny \cite{S} used this approach in studying the insertion-elimination Lie algebra $\g$, proving that $\g$ is simple as a Lie algebra and giving some fundamental results about representations for $\g$.  (We also note the papers \cite{KM04} and \cite{KM05} by Mencattini and Kreimer, investigating the ladder insertion-elimination algebra.  This Lie algebra was also inspired by \cite{CK} and can be characterized in terms of operations on trees; however the relations and resulting structure of the Lie algebra are quite different from the insertion-elimination algebra $\g$ under consideration in this paper.)  {The insertion-elimination algebra $\g$ is infinite-dimensional and $\z$-graded, and thus it bears some obvious similarity to the Virasoro algebra.  In the case of the Virasoro algebra, however, the weight spaces from the $\z$-grading are one-dimensional, whereas the weight spaces for the insertion-elimination algebra fail to have finite growth in the sense of \cite{Math92}.

The insertion-elimination algebra also possesses a triangular decomposition $\g = \n^- \oplus \h \oplus \n^+$ in the sense of \cite{MP95}.  This structure suggests certain classes of representations as natural objects of study, including Verma modules (the subject of study in \cite{S}) and Whittaker modules (the focus of this paper). Whittaker modules have been studied for a variety of Lie algebras with triangular decomposition (cf.  \cite{Ko78}, \cite{OW2008}, \cite{Ch08}, \cite{ZTL10}).   Additionally, Batra and Mazorchuk \cite{BatMaz} have developed a general framework for studying Whittaker modules, which captures all Lie algebras with triangular decomposition. 

This paper focuses on the search for simple Whittaker modules for the insertion-elimination algebra, a topic that cannot (currently) be resolved in the generality outlined in \cite{BatMaz}. In particular, we investigate the standard Whittaker module $M_\eta := U(\g) \otimes_{U(\n^+)} \cc_\eta$, where $\cc_\eta$ is the one-dimensional $\n^+$-module on which $\n^+$ acts by a homomorphism $\eta : \n^+ \to \cc$.  It is shown in \cite{BatMaz} that $M_\eta$ has a unique simple quotient $L(\mu, \eta)$ for each $\mu \in \h^+$. In Section \ref{sec:simplicityofM} of this paper we show that, under certain restrictions on $\eta$, the standard module $M_{\eta}$ itself is simple and thus is the only simple Whittaker module of ``type $\eta$." We note that the simplicity of $M_{\eta}$ provides further evidence for several conjectures made in \cite{BatMaz} about the nature of simple Whittaker modules for Lie algebras with triangular decomposition.  (See Conjectures 33, 34, and 40 in \cite{BatMaz}.) Batra and Mazorchuk \cite{BatMaz} also establish a connection between Verma modules and standard Whittaker modules; in Section \ref{sec:Vermas}, we exploit this connection to argue that all Verma modules for the insertion-elimination algebra admit the same central character.

\section{Definitions and notation}

\subsection{Rooted Trees}

The insertion-elimination algebra is defined in terms of operations on rooted trees. Therefore, we first establish a variety of notation and terminology describing rooted trees.  A {\it rooted tree} $t$ is an undirected, cycle-free graph with a distinguished vertex or {\it root}, denoted $\rt (t)$.  In this paper, rooted trees are displayed with the root at the top of the figure.

Let $\T$ denote the set of all rooted trees.  For $t \in \T$,  $V(t)$ is the set of vertices of $t$, $E(t)$ the set of edges of $t$, and $|t|$ the cardinality of $V(t)$.  For example, if $t$ is the rooted tree $\psset{levelsep=0.3cm, treesep=0.3cm} 
\pstree{\Tr{$\bullet$}}{\Tr{$\bullet$} \pstree{\Tr{$\bullet$}}{\Tr{$\bullet$} \Tr{$\bullet$} \Tr{$\bullet$}}}$, then $|t| = 6$ and $|E(t)| = 5$.   The \emph{root degree} of $t$, $\rootdeg (t)$, is the number of edges incident on $\rt (t)$. The \emph{depth} of $t$, $\depth (t)$, is the maximum length (i.e. number of edges) of any simple path (i.e. a path with no repeated vertices) originating at $\rt (t)$.  If $v \in V(t)$ is such that there is a path from $\rt (t)$ to $v$ that has length $\depth (t)$, then we say that $v$ has maximal distance from the root of $t$.  As an example, the depth of the following tree is 4, and there are two vertices that have maximal distance from the root. 
$$
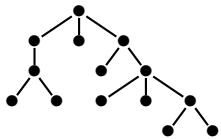
\pstree[levelsep=.4cm, treesep=.4cm]{\Tr{$\bullet$}}
{\pstree[]{\Tr{$\bullet$}}{\pstree[]{\Tr{$\bullet$}}{\Tr{$\bullet$} \Tr{$\bullet$}}} 
 \Tr{$\bullet$}  
 \pstree[]{\Tr{$\bullet$}}{\Tr{$\bullet$} \pstree[]{\Tr{$\bullet$}}{\Tr{$\bullet$} \Tr{$\bullet$} \pstree[]{\Tr{$\bullet$}}{\Tr{$\bullet$} \Tr{$\bullet$}}}}}
$$

For $s, t \in \T$ and $v \in V(s)$, let $s \cup_v t$ denote the rooted tree obtained by joining the root of $t$ to $s$ at the vertex $v$ via a single edge. We naturally identify $s$ and $t$ as subtrees of $s \cup_v t$, so that $V(s), V(t) \subseteq V(s \cup_v t)$ and  $E(s), E(t) \subseteq E(s \cup_v t)$.  For $t \in \T$ and $e \in E(t)$, removing the edge $e$ divides $t$ into two rooted trees $R_e(t)$, the ``root" subtree, and $P_e(t)$, the ``branch" subtree, so that $t= R_e(t) \cup_v P_e(t)$ for some $v \in V(R_e(t))$.  
  For $t_1, t_2, t_3 \in \T$, define 
\begin{align*}
\alpha (t_1, t_2, t_3) &= | \{ e \in E(t_2) \mid  R_e (t_2) = t_3, \ P_e(t_2) = t_1 \} |  \\
\beta (t_1, t_2, t_3) &= | \{ v \in V(t_3) \mid t_1 = t_3 \cup_v t_2 \} |.
\end{align*}

We call a multiset of rooted trees a {\it forest} and represent a forest two ways: as a collection $\lb t_1, \ldots, t_n \rb$ where the $t_i \in \T$ are not necessarily distinct; and as an ordered pair $(S, \m)$, where $S = \{s_1, \ldots, s_r \} \subseteq \T$ is the set of rooted trees in the forest and $\m: \T \rightarrow \z$ gives the multiplicity of each tree in the forest.  In particular, $\m(s)=0$ if $s \not\in S$, so we may regard $S$ as the set of rooted trees having nonzero multiplicity in the forest.
There is the following correspondence between the notations:
\begin{equation}\label{eqn:forestSetMult}
{\bf s}: (\{s_1, \ldots, s_k\}, \m) \leftrightarrow  \lb \underbrace{s_1, \ldots, s_1}_{\m(s_1)}, \underbrace{s_2, \ldots, s_2}_{\m(s_2)}, \ldots , \underbrace{s_k, \ldots, s_k}_{\m(s_k)}\rb.
\end{equation}
For a forest ${\bf s}= (\{s_1, \ldots, s_k\}, \m)=\lb t_1, \ldots, t_n \rb$, define the {\it length} of ${\bf s}$
$$
l({\bf s}) = \sum_{i=1}^k \m(s_i)=n;
$$
and the {\it order} of ${\bf s}$
$$
|{\bf s}| = \sum_{i=1}^k \m(s_i)|s_i|=\sum_{i=1}^n |t_i|.
$$

For a tree $t \in \T$, a vertex $v \in V(t)$, and a forest $\bf s$, let $t \cup_v {\bf s}$ be the rooted tree (unique up to isomorphism) created by joining each tree in $\bf s$ to the vertex $v$ via a new edge.   Using this notation, we define $\bigcup \bf s=\bullet \cup_\bullet {\bf s}$; that is, if ${\bf s} = \lb t_1, \ldots, t_n \rb$, then $\bigcup \bf s$ is the rooted tree created by joining a fixed vertex $v_0$ to the root of each $t_i$ via a new edge and declaring $\rt ( \bigcup {\bf s} ) = v_0$:

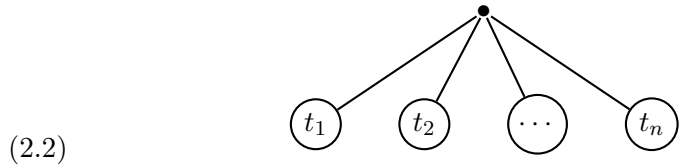
\begin{equation}
\pstree[levelsep=1.5cm]{\Tr{$\bullet$}}{\Tcircle{$t_1$} \Tcircle{$t_2$} \Tcircle{$\cdots$} \Tcircle{$t_n$}}
\end{equation}
Let $t \in \T$. If $t= \bigcup \lb t_1, \ldots, t_n \rb$ for some forest $\lb t_1, \ldots, t_n \rb$, then we refer to $t_1, \ldots, t_n$ as the {\it components} of $t$ and we define the forest of components $\C (t)$ by $\C (t)= \lb t_1, \ldots, t_n \rb$.

\subsection{The insertion-elimination algebra}\label{subsec:defsNotation}

The {\it insertion-elimination algebra} $\g$ is the Lie algebra over $\cc$ with basis $\{ d \} \cup \{ D_t^{\pm} \mid t \in \T \}$ and relations 
\begin{align*}
[D_s^+, D_t^+] &= \sum_{r \in \T} \left( \beta (r, s,t) - \beta (r,t,s)  \right) D_r^+  \\
&=  \sum_{v \in V(t)} D_{t \cup_v s}^+ - \sum_{v \in V(s)} D_{s \cup_v t}^+\\
[D_s^-, D_t^-] &= \sum_{r \in \T} \left( \alpha (t, r, s) - \alpha (s,t, r) \right) D_r^- \\
[D_s^-, D_t^+] &= \sum_{r \in \T} \alpha (s,t,r) D_r^+  + \sum_{r \in \T} \beta (s,t,r) D_r^-  \\
[D_t^-, D_t^+] &= d \\
[d, D_t^-] &= - |t| D_t^- \\
[d, D_t^+] &= |t| D_t^+,
\end{align*}
where $s, t \in \T$.  For example, $\psset{levelsep=0.3cm, treesep=0.3cm} [ D_{\pstree{\Tr{$\bullet$}}{}}^-, D_{\pstree{\Tr{$\bullet$}}{ \Tr{$\bullet$} \Tr{$\bullet$}}}^+ ] = 2 D_{\pstree{\Tr{$\bullet$}}{ \Tr{$\bullet$}}}^+$, 
$\psset{levelsep=0.3cm, treesep=0.3cm} [ D_{\pstree{\Tr{$\bullet$}}{ \Tr{$\bullet$} \Tr{$\bullet$}}}^-, D_{\pstree{\Tr{$\bullet$}}{}}^+ ] = D_{\pstree{\Tr{$\bullet$}}{ \Tr{$\bullet$}}}^-$ 
and
$$\psset{levelsep=0.3cm, treesep=0.3cm}
[ D_{\pstree{\Tr{$\bullet$}}{}}^+, D_{\pstree{\Tr{$\bullet$}}{ \Tr{$\bullet$} \Tr{$\bullet$}}}^+ ] 
= D_{\pstree{\Tr{$\bullet$}}{ \Tr{$\bullet$} \Tr{$\bullet$} \Tr{$\bullet$}}}^+ \quad + \quad 2 \, D_{\pstree{\Tr{$\bullet$}}{ \Tr{$\bullet$} \pstree{\Tr{$\bullet$}}{\Tr{$\bullet$}}}}^+ \quad - \quad D_{\pstree{\Tr{$\bullet$}}{\pstree{\Tr{$\bullet$}}{ \Tr{$\bullet$} \Tr{$\bullet$}}}}^+
$$
$$\psset{levelsep=0.3cm, treesep=0.3cm}
[ D_{\pstree{\Tr{$\bullet$}}{}}^-, D_{\pstree{\Tr{$\bullet$}}{ \Tr{$\bullet$} \Tr{$\bullet$}}}^- ] = 
-3 \, D_{\pstree{\Tr{$\bullet$}}{ \Tr{$\bullet$} \Tr{$\bullet$} \Tr{$\bullet$}}}^- \quad - \quad D_{\pstree{\Tr{$\bullet$}}{ \Tr{$\bullet$} \pstree{\Tr{$\bullet$}}{\Tr{$\bullet$}}}}^- \quad + \quad D_{\pstree{\Tr{$\bullet$}}{\pstree{\Tr{$\bullet$}}{ \Tr{$\bullet$} \Tr{$\bullet$}}}}^-
$$

Define the following subalgebras of $\g$:
\begin{align*}
\g_n &= \{ x \in \g \mid [d,x] = nx \}, \quad n \in \z; \\
\n^+ &= \{ D_t^+ \mid t \in \T \}= \bigoplus_{n>0} \g_n; \\
\n^- &= \{ D_t^- \mid t \in \T \}= \bigoplus_{n<0} \g_n; \\
\h &= \cc d= \g_0.
\end{align*}
Note that $\g = \n^+ \oplus \h \oplus \n^-$ and the subspaces $\g_n$ are weight spaces for the adjoint action of $\h$. Using the anti-involution given in \cite{OW-Aut}, this gives $\g$ a triangular decomposition in the sense of \cite[p.~95]{MP95}.  (The results in this paper do not depend on the anti-involution.)

Recall that for a given (ordered) basis for a Lie algebra, there exists a PBW basis for the universal enveloping algebra and hence for various induced modules.  With this in mind, we fix an (arbitrary) ordering $\prec$ on $\T$. Consider a forest ${\bf s}=(\{s_1, \ldots, s_k\}, \m)$, and assume without loss that $s_1 \succ \cdots \succ s_k$. Define
$$D_{\bf s}^-=(D_{s_1}^-)^{\m(s_k)} \cdots(D_{s_k}^-)^{\m(s_1)}, $$ 
and $D_{\emptyset}^-=1$.  (Equivalently, if we represent a forest as ${\bf s}=\lb t_1, \ldots, t_n \rb$, then we label the trees in the forest so that $t_1 \succeq \cdots \succeq t_n$, and we regard $D_{\bf s}^- = D_{t_1}^- \cdots D_{t_n}^-$.)  Then $\{ D_{\bf s}^- \mid {\bf s} \ \mbox{a forest} \}$ is basis for $U(\n^-)$ by the PBW theorem. Also, note that $U(\h)\cong \cc[d]$; that is, the elements of $U(\h)$ are polynomials in $d$, and thus we regard $\cc [d] \subseteq U( \g )$ when convenient.

\subsection{Whittaker modules and categories} \label{sec:Whittakerbackground}
The definitions, notation, and results in this section follow \cite{BatMaz}.

A Lie algebra $\n$ acts locally finitely on a module $V$ if $U(\n)v$ is finite-dimensional for every $v \in V$. The {\it Whittaker Category} $\W$ of $\g = \n^- \oplus \h \oplus \n^+$ is the full subcategory of $\g$-Mod on which the action of $\n^+$ is locally finite.

Because the insertion-elimination algebra $\g$ has a triangular decomposition, the pair $(\g, \n^+)$ is a Whittaker pair in the sense of \cite{BatMaz}.   In this context, the simple finite-dimensional modules are one-dimensional: they can be written as $\cc_\eta$ where $\eta \in (\n^+/[\n^+, \n^+])^*$ and $x.1 = \eta(x)$. Moreover, for $\eta \in (\n^+/[\n^+, \n^+])^*$, the set 
$$
V^{\eta} = \{ v \in V \mid (x-\eta(x))^kv=0 \ \mbox{for $x \in \n^+$ and $k>>0$}\}
$$  
is a $\g$-submodule of $V$, and any module $V \in \mbox{Ob} (\W)$ decomposes as 
$V= \bigoplus_{\eta} V^{\eta}$.   For modules $V, W \in \mbox{Ob} (\W)$, $\Hom_\g (V^\eta, W^\mu)=0$ unless $\eta = \mu$. Thus, modules $V$ in $\W$ can be understood via their components $V^{\eta}$; or equivalently, $\W$ can be understood via the blocks $\W_\eta$ where $\mbox{Ob} (\W_\eta)= \{ V \in \W \mid V=V^{\eta}\}$.  Given this, it is natural to investigate the {\it standard Whittaker module} $M_\eta$:
$$
M_\eta = U(\g) \otimes_{U(\n^+)} \cc_\eta \in \mbox{Ob} (\W_\eta ).
$$
For $V \in \W_\eta$, there is at least one {\it Whittaker vector} $v \in V$ such that $x.v=\eta(x)v$ for all $x \in \n^+$.  Therefore, by construction, $M_\eta$ surjects onto any simple Whittaker module in $\W_\eta$.  Thus it follows that if $M_\eta$ is simple, it is the unique simple object in $\W_\eta$.  This paper focuses on conditions that guarantee $M_\eta$ is simple.

Key to our arguments is the following result.

\begin{prop}[\cite{BatMaz}] \label{prop:Wdecomp}
 If a module $V \in \W$ contains a unique (up to scalar multiples) Whittaker vector, then $V$ is simple.
\end{prop}

\subsection{Lie algebra homomorphisms $\eta: \n^+ \rightarrow \cc$}

We investigate the simplicity of $M_\eta$ in Section \ref{sec:simplicityofM}.  Our results depend on the nature of $\eta: \n^+ \rightarrow \cc$.  With this in mind, we introduce several definitions to describe certain classes of Lie algebra homomorphisms $\n^+ \to \cc$.

A homomorphism $\eta : \n^+ \to \cc$ has \emph{finite support} if there is some positive integer $M$ such that $\eta (D_t^+) = 0$ whenever $|t| \ge M$. Otherwise, $\eta$ has {\it infinite support.} We also define and investigate two generalizations of finite support: depth-bounded homomorphisms and root-bounded homomorphisms.
A homomorphism  $\eta : \n^+ \to \cc$ is \emph{depth-bounded} if there exists a positive integer $M$ such that $\eta (D_t^+) = 0$ whenever $\depth (t) \ge M$.  A homomorphism $\eta : \n^+ \to \cc$ is \emph{root-bounded} if there exists a positive integer $M$ such that $\eta (t) = 0$ whenever $\rootdeg (t) \ge M$.

The problem of characterizing all possible Lie algebra homomorphisms $\eta$ is difficult due to the combinatorics involved in the Lie bracket in $\n^+$. Clearly, a map $\eta$ has finite support if and only if it is both depth-bounded and root-bounded.  However, it is easy to construct an example of a root-bounded homomorphism with infinite support.  Let $\ell_n$ represent the ladder with $n$ vertices, that is, the unique tree on $n$ vertices made up of a single path.  For each $n \in \z_{>0}$ choose $\eta_n \in \cc$, and define $\eta : \n^+ \to \cc$ by $\eta ( D_{\ell_n}^+ ) = \eta_n$ for all $n \ge 1$ and $\eta (D_t^+) = 0$ whenever $t$ is not a ladder.  If $\eta_n \neq 0$ for infinitely many $n$, then $\eta$ has infinite support, is not depth-bounded, and is root-bounded.  It is an open question if there exist homomorphisms with infinite support that are depth-bounded or homomorphisms with infinite support that are neither depth-bounded nor root-bounded.

\bigskip

\section{Simplicity of $M_{\eta}$} \label{sec:simplicityofM}

In this section, for $\eta \neq 0$, we show that $M_{\eta}$ is simple whenever $\eta$ is depth-bounded (Theorem \ref{thm:depthBndedSimp}) or root-bounded (Theorem \ref{thm:rootBoundedSimp}).   The corresponding result for the case that $\eta$ has finite support follows immediately from Theorem \ref{thm:depthBndedSimp}.  The basic approach for both results is to show that the space of Whittaker vectors in $M_{\eta}$ is one-dimensional, from which it follows that $M_\eta$ is simple by Proposition \ref{prop:Wdecomp}.

We outline the method of proof as follows.  We consider an arbitrary basis vector $D_{\bf s}^- p_{\bf s}(d) \otimes 1 \not\in \cc 1 \otimes 1$ of $M_{\eta}$ (where $D_{\bf s}^-$ is defined as in Section \ref{subsec:defsNotation}) and show that there is some $u \in \T$ such that
\begin{enumerate}
\item $(D_u^+-\eta(D_u^+)) D_{{\bf s}'}^- p_{{\bf s}'}(d) \otimes 1= 0$ for forests ${\bf s}'$ ``smaller" than ${\bf s}$ (for an ordering that depends on whether $\eta$ is depth-bounded or root-bounded), and 

\item $(D_u^+-\eta(D_u^+)) D_{\bf s}^- p_{\bf s}(d) \otimes 1 \neq 0$.  
\end{enumerate}

\noindent For $v \not\in \cc 1 \otimes 1$, $v$ may be written as 
\begin{equation}
v=\sum_{\mbox{${\bf s}$ is a forest}} D_{\bf s}^- p_{\bf s}(d) \otimes 1, \quad p_{\bf s}(d) \in U( \h ) \cong \cc [d].
\end{equation}
For such a $v$, there is a ``maximal" forest $\bf s$ such that $p_{\bf s}(d) \neq 0$; and so there is a corresponding $u \in \T$ such that $(D_u^+-\eta(D_u^+)) v \neq 0$.  Thus, $v$ is not a Whittaker vector.

Subsections \ref{sec:depthBounded} and \ref{sec:rootBounded} address the particulars of this argument, which depend on whether $\eta$ is depth-bounded or root-bounded.  First we establish some general computations.

The following lemma holds more generally than in the present context, and the proof is straightforward.

\begin{lem}
Let $x^+ \in U(\n^+)$ and $u \in U(\n^- \oplus  \cc d  )$, and let $\eta : \n^+ \to \cc$ be an algebra homomorphism.  Then in $M_\eta$, 
$$(x^+ - \eta (x^+) ) (u \otimes 1) = [x^+, u] \otimes 1.$$
\end{lem}

\begin{lem}\label{lem:D+p(d)v}
Let $\eta : \n^+ \to \cc$ be a Lie algebra homomorphism, and let $M_\eta$ be the standard $\eta$-Whittaker module.  For $p(d) \in \cc [d]$ and $t \in \T$, 
$$(D_t^+ - \eta (D_t^+)) (p(d) \otimes 1) = [D_t^+, p(d)] \otimes 1 = \eta (D_t^+) q(d) \otimes 1,$$
where $q(d) = p(d - |t|) - p(d)$.  
In particular, 
\begin{itemize}
\item[(i)] if $\eta (D_t^+) = 0$, then $ [D_t^+, p(d)] \otimes 1 = 0$;
\item[(ii)] if $deg (p(d))>0$ and $\eta(D_t^+) \neq 0$, then $ (D_t^+ - \eta (D_t^+)) (p(d) \otimes 1)\neq 0$.
\end{itemize}
\end{lem}

\begin{proof}
It is straightforward to check that $D_t^+ p(d) = p (d - |t|) D_t^+$, and therefore $[D_t^+, p(d)] = q(d) D_t^+$ where $q(d) = p(d - |t|) - p(d)$. Since $(D_t^+ - \eta (D_t^+)) (p(d) \otimes 1) = [D_t^+, p(d)] \otimes 1$, it then follows that $[D_t^+, p(d)] \otimes 1 = q(d) D_t^+ \otimes 1 = \eta (D_t^+) q(d) \otimes 1$ .
\end{proof}

\subsection{The homomorphism $\eta$ is depth-bounded.}\label{sec:depthBounded}

In this section, we show that $M_\eta$ is simple whenever $\eta \neq 0$ is depth-bounded.

\begin{lem}\label{lem:depthBounded-0}
Suppose that $\eta : \n^+ \to \cc$ is a non-zero depth-bounded algebra homomorphism.  Let $0<M \in \z$ such that $\eta (D_t^+) = 0$ whenever $\depth (t) > M$, let $t_0 \in \T$ with $\depth (t_0) = M$, and assume $v_0 \in V(t_0)$ has maximal distance from the root of $t_0$.   Suppose ${\bf s}_0$ and ${\bf s}$ are forests such that 
\begin{enumerate}
\item $|{\bf s}_0 | > | {\bf s}|$, or 
\item $| {\bf s}_0| = | {\bf s}|$ and $l ({\bf s}_0) > l ({\bf s})$, or 
\item $| {\bf s}_0| = | {\bf s}|$ and $l ({\bf s}_0) = l ({\bf s})$ and ${\bf s}_0 \neq {\bf s}$.
\end{enumerate}
Then for any polynomial $p(d)$,
$$[D_{t_0 \cup_{v_0} {\bf s}_0}^+,  D_{\bf s}^- p(d) ] \otimes 1 = 0.$$
\end{lem}
\begin{proof}
Note that any one of the assumptions (1), (2), or (3) implies that ${\bf s}_0$ is nonempty.  We proceed by induction on $l({\bf s})$. If $l({\bf s})=0$ (i.e. ${\bf s}$ is the empty forest), then the result follows from Lemma \ref{lem:D+p(d)v} since $\depth ( t_0 \cup_{v_0} {\bf s}_0 ) > \depth (t_0) = M$.  

Now suppose that $l({\bf s}) = k>0$, and write ${\bf s}= \lb s_1, \ldots, s_k \rb$.  For $\tilde {\bf s} = \lb s_2, \ldots, s_k \rb$ (where $\tilde {\bf s}$ may be empty), we have
\begin{align}
(D_{t_0 \cup_{v_0} {\bf s}_0}^+ - \eta (D_{t_0 \cup_{v_0} {\bf s}_0}^+)) (D_{\bf s}^- p(d) \otimes 1) &= [D_{t_0 \cup_{v_0} {\bf s}_0}^+, D_{s_1}^-  D_{\tilde {\bf s}}^- p(d)] \otimes 1 \nonumber \\
&= [D_{t_0 \cup_{v_0} {\bf s}_0}^+, D_{s_1}^- ] D_{\tilde {\bf s}}^- p(d) \otimes 1 \label{eq:depthbounded1} \\ 
& \quad + D_{s_1}^-  [D_{t_0 \cup_{v_0} {\bf s}_0}^+, D_{\tilde {\bf s}}^-  p(d) ] \otimes 1. \label{eq:depthbounded2}
\end{align}
Since $|{\bf s}_0 | > | \tilde {\bf s}|$ and $l ( \tilde {\bf s}) = k-1$, we conclude by induction that 
\newline $[D_{t_0 \cup_{v_0} {\bf s}_0}^+, D_{\tilde{\bf s}}^-  p(d) ] \otimes 1 = 0$  and thus  the term (\ref{eq:depthbounded2}) is also zero.

In (\ref{eq:depthbounded1}), we have $| t_0 \cup_{v_0} {\bf s}_0| > | {\bf s}_0| \ge |{\bf s}| \ge | s_1 |$; and therefore 
\newline 
$[D_{t_0 \cup_{v_0} {\bf s}_0}^+, D_{s_1}^- ] = -\sum_{r \in \T} \alpha (s_1, t_0 \cup_{v_0} {\bf s}_0, r) D_r^+$, 
where $\alpha (s_1, t_0 \cup_{v_0} {\bf s}_0, r)$ is the number of edges $e \in E(t_0 \cup_{v_0} {\bf s}_0)$ such that $R_e(t_0 \cup_{v_0} {\bf s}_0) = r$ and $P_e(t_0 \cup_{v_0} {\bf s}_0) = s_1$.   Clearly $\depth (r) > M$ whenever $\alpha (s_1, t_0 \cup_{v_0} {\bf s}_0, r) \neq 0$.  Therefore $D_r^+D_{\tilde {\bf s}}^- p(d) \otimes 1 = [D_r^+, D_{\tilde {\bf s}}^- p(d) ] \otimes 1$. Thus, to verify that (\ref{eq:depthbounded1}) is zero, it suffices to show that 
$[D_r^+, D_{\tilde {\bf s}}^- p(d) ] \otimes 1 = 0$
whenever $r = R_e ( t_0 \cup_{v_0} {\bf s}_0)$ and $s_1 = P_e(t_0 \cup_{v_0} {\bf s}_0)$.  To prove this, we consider the following three cases for the edge $e$.  
\begin{enumerate}
\item[i.] Suppose $e \in E(t_0) \subseteq E( t_0 \cup_{v_0} {\bf s}_0)$.  Clearly $e$ cannot belong to the simple path from $\rt  (t_0)$ to $v_0$ as this would force $| P_e( t_0 \cup_{v_0} {\bf s}_0 )| \ge 1 + | {\bf s}_0| > |s_1|$.   This implies that $r$ has the form $t_0' \cup_{v_0} {\bf s}_0$, where $t_0' = R_e(t_0)$ and $\depth (t_0') = \depth (t_0) = M$.  
By assumption, $|{\bf s}_0| \ge |{\bf s}|$, so that $| {\bf s}_0 | > | {\bf s} | - |s_1| = | \tilde {\bf s} |$.  Therefore the inductive hypothesis implies $[D_{t_0' \cup_{v_0} {\bf s}_0}^+, D_{\tilde {\bf s}}^- p(d) ] \otimes 1 = 0$.

\smallskip
\item[ii.] Suppose $e$ is an edge joining $v_0$ to the root of one of the trees in ${\bf s}_0$.  In this case, we may write $r = t_0 \cup_{v_0} \tilde {\bf s}_0$, where $\tilde {\bf s}_0$ is a forest formed from $t_0 \cup_{v_0} {\bf s}_0$ by removing a copy of $s_1$ from ${\bf s}_0$.   Note that $| \tilde {\bf s}_0 | = | {\bf s}_0| - | s_1|$ and $| \tilde {\bf s} | = | {\bf s}| - |s_1|$, and it is straightforward to see that $\tilde {\bf s}_0$ and $\tilde {\bf s}$ fall under one of assumptions (1)-(3).  Therefore it follows from induction that $[D_{t_0 \cup_{v_0} \tilde {\bf s}_0}^+, D_{\tilde {\bf s}}^- p(d) ] \otimes 1 = 0$.

\smallskip
\item[iii.] Suppose $e \in E(s) \subseteq E(t_0 \cup_{v_0} {\bf s}_0)$ for some $s \in {\bf s}_0$.  In this case, we may write $r = t_0 \cup_{v_0} {\bf s}_0'$, where ${\bf s}_0'$ is a forest formed by replacing the tree $s \in {\bf s}_0$ with $R_e(s)$ (which is necessarily nonempty).   Note that $| {\bf s}_0' | = | {\bf s}_0| - | s_1|$ and $| \tilde {\bf s} | = | {\bf s} | - |s_1|$, and it is straightforward to show that $ {\bf s}_0', {\bf s}_0$ fall under one of assumptions (1)-(2).  Therefore it follows from induction that $ [D_{t_0 \cup_{v_0} {\bf s}_0'}^+, D_{\tilde {\bf s}}^- p(d) ] \otimes 1 = 0$. 
\end{enumerate}
\end{proof}

For a nonempty forest ${\bf s}=(S, \m)$, let 

$$\Pi ({\bf s})=\prod_{i=1}^k \m(s_i) !,$$ 
where $|S|=k$. That is, $\Pi ({\bf s})$ is  the product of the factorials of the multiplicities of the trees in ${\bf s}$.   We will use the following result in the special case that $\eta (D_{t_0}^+) \neq 0$.

\begin{lem}\label{lem:depthBounded-non0}
Suppose that $\eta : \n^+ \to \cc$ is a non-zero depth-bounded algebra homomorphism.  Let $M \in \z$ be minimal such that $\eta (D_t^+) = 0$ whenever $\depth (t) > M$.  Let $t_0 \in \T$ with $\depth (t_0) = M$, and assume $v_0 \in V(t_0)$ has maximal distance from $\rt (t_0)$.   Then for a nonempty forest ${\bf s}_0$ and a polynomial $p(d)$,
$$[D_{t_0 \cup_{v_0} {\bf s}_0}^+,  D_{{\bf s}_0}^-  p(d)] \otimes 1 = (-1)^{l({\bf s}_0)} \Pi ({{\bf s}_0}) D_{t_0}^+ p(d) \otimes 1.$$
\end{lem}
\begin{proof}
Since $\depth ( t_0 \cup_{v_0} {\bf s}_0) > M$, we have that $\eta ( D_{t_0 \cup_{v_0} {\bf s}_0}^+) = 0$ and by Lemma \ref{lem:D+p(d)v}, 
$$[D_{t_0 \cup_{v_0} {{\bf s}_0}}^+,  D_{{\bf s}_0}^-  p(d)]  \otimes 1=[D_{t_0 \cup_{v_0} {\bf s}_0}^+,  D_{{\bf s}_0}^- ] p(d) \otimes 1.$$

We now proceed by induction on $l({\bf s}_0)$.  For the base case, let ${\bf s}_0 = \lb s \rb$ for some $s \in \T$.  Then
\begin{align*}
[D_{t_0 \cup_{v_0} {\bf s}_0}^+,  D_{{\bf s}_0}^- ] p(d) \otimes 1 &= [ D_{t_0 \cup_{v_0} s}^+, D_s^-] p(d) \otimes 1\\
&= - \sum_{r \in \T} \alpha (s, t_0 \cup_{v_0} s, r) D_r^+ p(d) \otimes 1,
\end{align*}
 where $\alpha (s, t_0 \cup_{v_0} s, r)$ is the number of edges $e \in E( t_0 \cup_{v_0} s)$ such that $R_e (t_0 \cup_{v_0} s) = r$ and $P_e(t_0 \cup_{v_0} s) = s$.  If $e \in E(s) \subseteq E( t_0 \cup_{v_0} s)$, we get that $|P_e(t_0 \cup_{v_0} s)|<|s|$ and so $P_e(t_0 \cup_{v_0} s) \neq s$.  If $e \in E(t_0) \subseteq E(t_0 \cup_{v_0} s)$ is on a simple path from the root of $t_0$ to $v_0$, we get that $|P_e(t_0 \cup_{v_0} s)|>|s|$ and so $P_e(t_0 \cup_{v_0} s) \neq s$. For $e \in E(t_0) \subseteq E(t_0 \cup_{v_0} s)$ not on a path from the root of $t_0$ to $v_0$, we get that $\depth (r)> \depth (t_0)=M$ and so $D_r^+ p(d) \otimes 1=0$.  Thus the only possible nonzero terms in the sum correspond to the case that $e$ is the edge of $t_0 \cup_{v_0} s$ joining $v_0$ to the root of $s$, and $[ D_{t_0 \cup_{v_0} s}^+, D_s^-] = - D_{t_0}^+$.  Then 
$$[ D_{t_0 \cup_{v_0} s}^+, D_s^-] p(d) \otimes 1 =  -D_{t_0}^+ p(d) \otimes 1=-\Pi ( \lb s \rb ) D_{t_0}^+ p(d) \otimes 1.$$
This completes the base case.

Now suppose that $k = l({\bf s}_0) \geq 2$. Write ${{\bf s}_0} = \lb s_1, \ldots, s_k \rb$and  $\tilde {\bf s}_0 = \lb s_2, \ldots, s_k \rb \neq \emptyset$. Then
\begin{align*}
[D_{t_0 \cup_{v_0} {{\bf s}_0}}^+,  D_{{\bf s}_0}^-  p(d)]  \otimes 1 
&= [D_{t_0 \cup_{v_0} {\bf s}_0}^+, D_{s_1}^- ] D_{\tilde {\bf s}_0}^- p(d) \otimes 1  \\ 
& \quad + D_{s_1}^-  [D_{t_0 \cup_{v_0} {\bf s}_0}^+, D_{\tilde {\bf s}_0}^-  p(d) ] \otimes 1\\
&=[D_{t_0 \cup_{v_0} {\bf s}_0}^+, D_{s_1}^- ] D_{\tilde {\bf s}_0}^- p(d) \otimes 1
\end{align*}
since $l({\bf s}_0 ) > l( \tilde {\bf s}_0)$ and Lemma \ref{lem:depthBounded-0} thus implies $[D_{t_0 \cup_{v_0} {\bf s}_0}^+, D_{\tilde {\bf s}_0}^-  p(d) ] \otimes 1 = 0$.

Since $| t_0 \cup_{v_0} {\bf s}_0| > | {\bf s}_0| \ge |s_1|$, we have $[D_{t_0 \cup_{v_0} {\bf s}_0}^+, D_{s_1}^- ] = -\sum_{r \in \T} \alpha (s_1, t_0 \cup_{v_0} {\bf s}_0, r) D_r^+$, where $\alpha (s_1, t_0 \cup_{v_0} {\bf s}_0, r) $ is the number of edges $e \in E(t_0 \cup_{v_0} {\bf s}_0)$ such that $R_e(t_0 \cup_{v_0} {\bf s}_0) = r$ and $P_e( t_0 \cup_{v_0} {\bf s}_0 ) = s_1$.  Since $l ({\bf s}_0) > 1$, it follows that $\depth (r) > M$ and $\eta (D_r^+) = 0$ whenever $\alpha (s_1, t_0 \cup_{v_0} {\bf s}_0, r)  \neq 0$. Therefore, $D_r^+D_{\tilde {\bf s}_0}^- p(d) \otimes 1 = [D_r^+, D_{\tilde {\bf s}_0}^- p(d) ] \otimes 1 = [D_r^+, D_{\tilde {\bf s}_0}^- ] p(d) \otimes 1$, and we proceed by showing that 
$$[D_r^+, D_{\tilde {\bf s}_0}^- ] p(d) \otimes 1= 0$$
whenever $r = R_e ( t_0 \cup_{v_0} {\bf s}_0)$ and $s_1 = P_e( t_0 \cup_{v_0} {\bf s}_0)$, except in the case that $e$ joins the vertex $v_0$ to a tree in ${\bf s}_0$ that is isomorphic to $s_1$.  We consider the following cases for $e$:
\begin{enumerate}
\item[i.] Suppose $e \in E(t_0) \subseteq E( t_0 \cup_{v_0} {\bf s}_0)$.   If $e$ belongs to a simple path from the root of $t_0$ to $v_0$, then $| P_e (t_0 \cup_{v_0} {\bf s}_0 ) | > |s_1|$, which is impossible.  Thus $r$ has the form $t_0' \cup_{v_0} {\bf s}_0$, where $t_0' = R_e(t_0)$ and $\depth (t_0') = \depth (t_0)$.   Now $\depth (r) = \depth (t_0 \cup_{v_0} {\bf s}_0)$, so Lemma \ref{lem:depthBounded-0} implies $[D_r^+, D_{\tilde {\bf s}_0}^- p(d) ] \otimes 1 = 0$ since $l({\bf s}_0) > l(\tilde {\bf s}_0)$.

\item[ii.] Suppose $e \in E(s) \subseteq E(t_0 \cup_{v_0} {\bf s}_0)$ for some $s \in {\bf s}_0$ with $P_e(s) = s_1$.  In this case, we may write $r = t_0 \cup_{v_0} {\bf s}_0'$, where ${\bf s}_0'$ is a forest formed by replacing some $s \in {\bf s}_0$ with $R_e(s)$.   Then $| {\bf s}_0'| \ge | \tilde {\bf s}_0 |$ and $l( {\bf s}_0')  = l({\bf s}_0) > l(\tilde {\bf s}_0)$, so $[D_r^+, D_{\tilde {\bf s}_0}^- p(d) ] \otimes 1 = 0$ by Lemma \ref{lem:depthBounded-0}.

\item[iii.] Suppose $e$ is an edge joining $v_0 \in V(t_0)$ to $s_1$ (or a tree isomorphic to $s_1$ in ${\bf s}_0$).  Then $r = t_0 \cup_{v_0} \tilde {\bf s}_0$, and $\alpha (s_1, t_0 \cup_{v_0} {\bf s}_0, r)$ is exactly the multiplicity $\m_1$ of $s_1$ in ${\bf s}_0$.  Thus 
\begin{align*}
[D_{t_0 \cup_{v_0} {\bf s}_0}^+,  D_{{\bf s}_0}^- ] p(d) \otimes 1 &= -\m_1 D_{t_0 \cup_{v_0} \tilde {\bf s}_0}^+ D_{\tilde {\bf s}_0}^- p(d) \otimes 1 \\
&= -\m_1 [D_{t_0 \cup_{v_0} \tilde {\bf s}_0}^+, D_{\tilde {\bf s}_0}^- p(d) ] \otimes 1.
\end{align*}
Now $l(\tilde {\bf s}_0) < l ({\bf s}_0)$ and the multiplicity of $s_1$ in $\tilde{\bf s}_0$ is $\m_1-1$, so the result holds by induction.
\end{enumerate}
\end{proof}

\begin{thm}\label{thm:depthBndedSimp}
Let $\eta$ be a non-zero, depth-bounded homomorphism.  Then the space of Whittaker vectors of type $\eta$ in $M_\eta$ is one dimensional, and $M_\eta$ is simple.
\end{thm}
\begin{proof}
Let $v \in M_\eta$ with $v \not\in \cc 1 \otimes 1$; we show that $v$ is not a Whittaker vector. Write $v = \sum_{{\bf s}}  D_{\bf s}^- p_{\bf s}(d)\otimes 1$, and note that by Lemma \ref{lem:D+p(d)v} we may assume that there exists a forest ${\bf s}$ such that ${\bf s} \neq \emptyset$ and $p_{\bf s}(d) \neq 0$.

Let $M \in \z$ be minimal such that $\eta (D_t^+) = 0$ whenever $\depth (t) > M$, and let $t_0 \in \T$ with $\depth (t_0) = M$ and $\eta ( D_{t_0}^+ ) \neq 0$.  Assume $v_0 \in V(t_0)$ has maximal distance from the root of $t_0$.  Define $N = \max \{ | {\bf s}| \mid p_{\bf s}(d) \neq 0 \}>0$ and $\Lambda_N = \{ {\bf s} \mid p_{\bf s}(d) \neq 0, | {\bf s}| = N \}$, and choose ${\bf s}_0 \in \Lambda_N$ with $l ({\bf s}_0)$ maximal among elements of $\Lambda_N$. By Lemma \ref{lem:depthBounded-0} and Lemma \ref{lem:depthBounded-non0}, 
\begin{align*}
\left( D_{t_0 \cup_{v_0} {\bf s}_0}^+ - \eta ( D_{t_0 \cup_{v_0} {\bf s}_0}^+) \right) v &=  \sum_{\bf s} [ D_{t_0 \cup_{v_0} {\bf s}_0}^+ , D_{\bf s}^- p_{\bf s}(d)] \otimes 1  \\
&= (-1)^{l ({\bf s}_0)} \Pi ({\bf s}_0) D_{t_0}^+ p_{{\bf s}_0}(d) \otimes 1.
\end{align*}
Then by Lemma \ref{lem:D+p(d)v}, 
\begin{align*}
D_{t_0}^+ p_{{\bf s}_0}(d) \otimes 1 &= [D_{t_0}^+, p_{{\bf s}_0}(d)] \otimes 1 + \eta ( D_{t_0}^+ ) ( p_{{\bf s}_0}(d)) \otimes 1\\
&= \eta ( D_{t_0}^+ ) ( q(d) + p_{{\bf s}_0}(d)) \otimes 1\\
&=  \eta ( D_{t_0}^+ ) p_{{\bf s}_0}( d - |t_0|)  \otimes 1.
\end{align*}
Thus $$(D_{t_0 \cup_{v_0} {\bf s}_0}^+ - \eta (D_{t_0 \cup_{v_0} {\bf s}_0}^+))v = (-1)^{l ({\bf s}_0)}  \Pi ({\bf s}_0) \eta ( D_{t_0}^+) p_{{\bf s}_0}( d - |t_0|) \otimes 1.$$ In particular, $\left( D_{t_0 \cup_{v_0} {\bf s}_0}^+ - \eta ( D_{t_0 \cup_{v_0} {\bf s}_0}^+) \right) v \neq 0$; that is, $v$ is not a Whittaker vector.  The rest of result follows from Proposition \ref{prop:Wdecomp}.
\end{proof}

Every homomorphism $\eta : \n^+ \to \cc$ with finite support is clearly depth-bounded, so we also have the following.

\begin{cor} \label{cor:finitesupport}
If $\eta : \n^+ \to \cc$ has finite support and is nonzero, then $M_\eta$ is simple and has a one-dimensional space of Whittaker vectors.
\end{cor}

\subsection{The homomorphism $\eta$ is root-bounded} \label{sec:rootBounded}
We now study $M_\eta$ in the case that $\eta : \n^+ \to \cc$ is root bounded and nonzero.  In light of Corollary \ref{cor:finitesupport}, in this section we focus on root-bounded $\eta$ with infinite support.  We first introduce some new notation and present several computational lemmas; the main results, following the argument sketched at the beginning of the section, appear in Lemma \ref{lem:actionOrderRootBnd} and Theorem \ref{thm:rootBoundedSimp}.

Suppose that $\eta$ is root-bounded, and let $0 < R \in \z$ (not necessarily minimal) such that $\eta (D_t^+) = 0$ whenever $\rootdeg (t) > R$.  For each $n \in \z_{> 0}$, define $\mathcal S_n = \{ |t| \mid \rootdeg (t) = n, \eta (D_t^+) \neq 0 \}$. Then let $\omega_n \in \z_{\ge 0} \cup \{ \infty \}$  be defined by
$$\omega_n = \left\{ \begin{array}{ll}  0 & \text{if $\mathcal S_n= \emptyset$ (i.e. if $\eta (D_t^+) = 0$ whenever $\rootdeg (t) = n$),}  \\ 
M & \text{if $\mathcal S_n \neq \emptyset$ has a maximal element $M \in \z$,} \\ 
\infty & \text{if $\mathcal S_n \neq \emptyset$ does not contain a maximal element in $\z$.}  \end{array} \right.$$
For $\omega_n <\infty$, we may regard $\omega_n$ as the size of the largest tree $t$ of root degree $n$ such that $\eta (D_t^+) \neq 0$.  In particular, $\omega_n = 0$ whenever $n > R$.

If $\eta$ is root-bounded with infinite support, there is some $n > 0$ such that $\omega_n = \infty$. On the other hand, the set $\{ n \mid \omega_n = \infty \}$ is finite since $\omega_n = 0$ for $n >R$.  Thus we may define 
\begin{equation}\label{eqn:def_Reta}
R_\eta = \max \{ n \mid \omega_n = \infty \}.
\end{equation}
By definition, we have that $\omega_n$ is finite if $n > R_\eta$.  Moreover, the set $\{ \omega_n \mid n > R_\eta \} \subseteq \z$ is finite, so we may define
\begin{equation}\label{eqn:def_Beta}
B_\eta = \max \{ \omega_n \mid n > R_\eta \}.
\end{equation}
Thus if $t \in \T$ with $\rootdeg (t) > R_\eta$ and $|t| > B_\eta$, it follows that $\eta (D_t^+) = 0$.  However, since $\omega_{R_\eta} = \infty$, we can choose $t \in \T$ with $\rootdeg (t) = R_\eta$ and $\eta (D_t^+) \neq 0$ such that $|t|$ is arbitrarily large.

\begin{lem} \label{lem:bigRootDegZero}
Fix a forest ${\bf s}$ and a polynomial $p(d)$.  If $u \in \T$ is such that $\rootdeg(u)>R_{\eta}+l({\bf s})$ and $|u|>B_{\eta}+|{\bf s}|$, then $D_u^+ D_{\bf s}^- p(d) \otimes 1 =[D_u^+ ,D_{\bf s}^- p(d)] \otimes 1 = [D_u^+ ,D_{\bf s}^- ] p(d) \otimes 1 = 0$. 
\end{lem}
\begin{proof}
Note that  $\eta(D_u^+)=0$ by our choice of $u$, so $D_u^+ D_{\bf s}^- p(d) \otimes 1 =[D_u^+ ,D_{\bf s}^- p(d)] \otimes 1 = [D_u^+ ,D_{\bf s}^- ] p(d) \otimes 1$.

The proof is by induction on $l({\bf s})$.  If $l({\bf s}) = 0$, then ${\bf s}$ is the empty forest and we have $D_u^+ p(d) \otimes 1=0$ by Lemma \ref{lem:D+p(d)v}.  Now assume $l({\bf s}) > 0$. Write ${\bf s} = \lb s_1, \ldots, s_k \rb$.  For $\tilde {\bf s} = {\bf s}\setminus \lb s_1 \rb$,
$$ 
[D_u^+,  D_{\bf s}^-] p(d)\otimes 1 =  [ D_u^+,  D_{s_1}^-] D_{\tilde {\bf s}}^- p(d)\otimes 1 + D_{s_1}^- [ D_u^+,  D_{\tilde {\bf s}}^- p(d)]\otimes 1.
$$
Since $l(\tilde {\bf s}) < l({\bf s})$, we have that $[ D_u^+,  D_{\tilde {\bf s}}^-  p(d)]\otimes 1 = 0$ by induction.   Also, $|u| > B_\eta + |{\bf s}| \ge |{\bf s}| \ge |s_1|$. This implies $[ D_u^+,  D_{s_1}^-] = -\sum_{u'} \alpha(s_1, u, u') D_{u'}^+$, where $\alpha(s_1, u, u')$ is the number of edges $e \in E(u)$ such that $R_e(u) = {u'}$ and $P_e(u) = s_1$.  For $u'$ such that $\alpha(s_1, u, u') \neq 0$, we have
$$\rootdeg ({u'}) \ge \rootdeg (u) - 1 > R_\eta + l({\bf s}) -1 = R_\eta + l( \tilde {\bf s}).$$
Moreover, $|{u'}| = |u| - |s_1| > B_\eta + |{\bf s}| - |s_1| = B_\eta + | \tilde {\bf s}|$. It then follows from induction that $D_{u'}^+ D_{\tilde {\bf s}}^- p(d)\otimes 1= [D_{u'}^+, D_{\tilde {\bf s}}^-] p(d) = 0$ whenever $\alpha (s_1, u, u') \neq 0$.  Therefore 
$$[D_u^+,  D_{\bf s}^- p(d)]\otimes 1 =  [ D_u^+,  D_{s_1}^-] D_{\tilde {\bf s}}^- p(d) \otimes 1 = \sum_{u'} \alpha (s_1, u, u') D_{u'}^+ D_{\tilde {\bf s}}^- p(d) \otimes 1 = 0,$$
as desired.
\end{proof}

\begin{lem} \label{lem:Ordering}
Assume $\eta$ is root-bounded with infinite support, and let ${\bf s}$ and ${\bf s}'$ be forests with $l ({\bf s}) = l({\bf s}') = k$ and ${\bf s}'=\lb s_1', \ldots, s_k' \rb$.   Let $t_0 \in \T$ such that $\rootdeg (t_0) = R_\eta$, and $|t_0| > B_\eta + |{\bf s}'|$, and let $p(d) \in \cc[d]$.  Then for a permutation $\pi$ of $\{ 1, \ldots, k \}$, 
$$[D_{t_0 \cup_{\rt (t_0)} {\bf s}}^+, D_{{\bf s}'}^- p(d)] \otimes 1 = [D_{t_0 \cup_{\rt (t_0)} {\bf s}}^+, D_{s_{\pi(1)}'}^- D_{s_{\pi(2)}'}^- \cdots D_{s_{\pi(k)}'}^-  p(d)] \otimes 1.$$
\end{lem}
\begin{proof}
Note that
$$D_{{\bf s}'}^-  = D_{s_{\pi(1)}'}^- D_{s_{\pi(2)}'}^- \cdots D_{s_{\pi(k)}'}^- + \sum_{l({{\bf s}''}) < k} \gamma_{{\bf s}''} D_{{\bf s}''}^-$$
for $\gamma_{{\bf s}''} \in \cc$.  By Lemma \ref{lem:bigRootDegZero}, $[D_{t_0 \cup_{\rt (t_0)} {\bf s}}^+, D_{{\bf s}''}^- p(d)] \otimes 1 = 0$ since $l({{\bf s}''}) < l({\bf s})$.
\end{proof}

Recall that for $t \in \T$, $\C (t)$ represents the forest of components of $t$.

\begin{lem}\label{lem:partReplacement}
Assume $\eta$ is root-bounded with infinite support.   Let ${\bf s}, {{\bf s}'}$ be forests of $k$ rooted trees, $p(d)$ a polynomial in $d$, and $t \in \T$ with $\rootdeg (t) = R_\eta$ and $|t| > B_\eta + |{{\bf s}'}|$.  Let ${\bf w} = {\bf s} \cap {{\bf s}'}$.  
\begin{itemize}
\item[(i)]  If ${{\bf s}'}\setminus {\bf w} \subseteq \C (t)$, then 
$$[D_{t \, \cup_{\rt (t)} {\bf s}}^+, D_{{\bf s}'}^- p(d)] \otimes 1 = \xi \eta (D_{u_0}^+) q(d) \otimes 1,$$
where $0 \neq \xi \in \z$, $q(d)=p (d - (|t|+|{\bf s}|-|{\bf s}'|)) - p(d)$, and 
$$
u_0 =\bigcup_{r \in (\C (t) \setminus({{\bf s}'} \setminus {\bf w})) \cup({\bf s} \setminus {\bf w})} r. 
$$ 
In particular, $[D_{t \, \cup_{\rt (t)} {\bf s}}^+, D_{\bf s}^- p(d)] \otimes 1 = \xi \eta (D_{t}^+) q(d) \otimes 1$, where  $q(d)=p (d - |t|) - p(d)$.
\item[(ii)] If ${{\bf s}'}\setminus {\bf w} \not\subseteq \C (t)$, then $[D_{t \, \cup_{\rt (t)} {\bf s}}^+, D_{{\bf s}'}^- p(d)] \otimes 1 = 0$.
\end{itemize}
\end{lem}

\begin{proof}
Let $u=t \cup_{\rt (t)} {\bf s}$.  Write ${\bf s} = \lb s_1, \ldots, s_k \rb$ and ${{\bf s}'} = \lb s_1', \ldots, s_k' \rb$.

We first reduce to the case ${\bf w}= \emptyset$.   Suppose ${\bf w}$ is nonempty.  Using Lemma \ref{lem:Ordering}, we may assume that $s_1=s_1'$.  Write $\tilde {\bf s}  = {\bf s}\setminus \lb s_1 \rb$ and $\tilde {\bf s}'  = {\bf s}'\setminus \lb s_1 \rb$.  Using $\eta ( D_u^+) = 0$, we have
\begin{align}
 [D_u^+,  D_{{\bf s}'}^-p(d)] \otimes 1 
&=  [ D_u^+,  D_{s_1}^-] D_{\tilde {\bf s}'}^- p(d) \otimes 1 \label{eqn:partReplacement1} \\
& \ \ \ + D_{s_1}^- [ D_u^+,  D_{\tilde {\bf s}'}^-p(d) ] \otimes 1. \label{eqn:partReplacement2}
\end{align}
Since $\rootdeg (u) = R_\eta + l({\bf s}) > R_\eta + l( \tilde {\bf s}')$ and $|u| > |t| > B_\eta + |{\bf s}'| > B_\eta + | \tilde {\bf s}'|$, Lemma \ref{lem:bigRootDegZero} implies that  $[ D_u^+,  D_{\tilde {\bf s}'}^- p(d)] \otimes 1 = 0$ and hence (\ref{eqn:partReplacement2}) is zero.  

Now $|u| >  |s_1|$ and so $[ D_u^+,  D_{s_1}^-] =- \sum_{u'} \alpha(s_1, u, u') D_{u'}^+$, where $\alpha(s_1, u, u')$ is the number of edges $e \in E(u)$ such that $R_e(u) = {u'}$ and $P_e(u) = s_1$.  Clearly $\rootdeg ({u'}) \ge \rootdeg (u) - 1$ whenever $\alpha(s_1, u, u') \neq 0$.  Suppose that $e$ is an edge that is not incident on $\rt (t)$.  Then $\rootdeg ({u'}) = \rootdeg (u) = R_\eta + l({\bf s}') > R_\eta + l( \tilde {\bf s}')$ and $|{u'}| = |u| - |s_1| = |t_0| + | \tilde {\bf s}'| > B_\eta + |{\bf s}'| - |s_1| = B_\eta + |\tilde {\bf s}'|$.  In this case, Lemma \ref{lem:bigRootDegZero} implies that $D_{u'}^+ D_{\tilde {\bf s}'}^- p(d)\otimes 1 = 0$.  Thus 
$$[ D_u^+,  D_{s_1}^-] D_{\tilde {\bf s}'}^- p(d)\otimes 1 = -\alpha(s_1, u, \tilde u) D_{\tilde u}^+ D_{\tilde {\bf s}'}^- p(d)\otimes 1,$$
where $\tilde u = t \cup_{\rt (t)} \tilde {\bf s}$ and $\alpha(s_1, u, \tilde u)$ is the sum of the number of components of $t$ isomorphic to $s_1$ and the number of copies of $s_1$ in ${\bf s}$.  In particular, $\alpha(s_1, u, \tilde u) \geq 1$.
 By repeatedly applying this argument, we reduce to the case where ${\bf w}= \emptyset$.

We now consider the situation where ${\bf s}$ and ${{\bf s}'}$ share no common elements. We proceed by induction on $k = l({\bf s}) = l({\bf s}')$, noting that the case $k=0$ follows from Lemma \ref{lem:D+p(d)v}.  Therefore, suppose $k>0$. Using the same arguments as the first part of the proof, we have
\begin{align*}
[D_u^+, D_{{{\bf s}'}}^- p(d)] \otimes 1 
&= -\sum_{u'} \alpha(s_1', u, u') D_{u'}^+ D_{\tilde {\bf s}'}^- p(d)\otimes 1,
\end{align*}
where $\alpha(s_1', u, u')$ is the number of edges $e \in E(u)$ such that $R_e(u) = u'$ and $P_e(u) = s_1'$.   Note that if $e$ is not incident on $\rt (t)$, then $\rootdeg (u') = \rootdeg (t) + l({\bf s}) > R_\eta + l( \tilde {\bf s}')$ and
$$|u'| = |u| - |s_1'| > B_\eta + |{\bf s}| + |{{\bf s}'}| - | s_1'| = B_\eta + |{\bf s}| + | \tilde {\bf s}'|.$$
In this case, Lemma \ref{lem:bigRootDegZero} implies that $D_{u'}^+ D_{\tilde {\bf s}'}^- p(d)\otimes 1 = 0$.  

Now it remains to consider $u'=R_e(u)$ where $e$ is incident on $\rt (t)$.  Since ${\bf s}' \cap {\bf s} = \emptyset$, it follows that $\alpha(s_1', u, u') \neq 0$ only if $t$ contains a component isomorphic to $s_1'$ and $e$ is an edge connecting such a component of $t$ to the root $\rt (t)$.  Write $\C (t) = \lb t_1, \ldots, t_{R_\eta} \rb$ and assume (without loss) that $s_1' \cong t_1$, so that $u' = \tilde t \cup_{\rt (\tilde t)} \tilde {\bf s}$, where $\tilde t$ is the rooted tree formed from $t$ by replacing the component $t_1$ with $s_1$.  We now have 
$$D_{u'}^+ D_{\tilde {\bf s}'}^- p(d) \otimes 1 = D_{\tilde t \cup_{\rt (\tilde t)} \tilde {\bf s}}^+ D_{\tilde {\bf s}'}^- p(d) \otimes 1 = [D_{\tilde t \, \cup_{\rt (\tilde t)} \tilde {\bf s}}^+, D_{\tilde {\bf s}'}^- p(d)] \otimes 1,$$
where $\rootdeg (\tilde t) = R_\eta$, $|\tilde t| > B_\eta + |{\tilde {\bf s}'}|$.  The result now follows by induction.
\end{proof}

We require some new notation for Lemma \ref{lem:actionOrderRootBnd} and Theorem \ref{thm:rootBoundedSimp}.  Consider a forest ${\bf s}=(S, \m)$. Define $\iota ({\bf s})=|S|$ and write $S=\{s_1, \ldots, s_{\iota ({\bf s})} \}$ where we assume that $\m(s_1) \geq \m(s_2) \ge \cdots \ge \m(s_{\iota ({\bf s})})$.  This defines a partition $\lambda ({\bf s}) = ( \m(s_1), \ldots, \m(s_{\iota ({\bf s})}))$ of the integer $l ({\bf s})$.  We view partitions as comparable via the standard lexicographic order $\leq$.  Note that different forests can give rise to the same partition.   However, for a forest ${\bf s}'=(S', \m')$, if ${\bf s} \neq {\bf s}'$ and $\lambda ({\bf s}) \ge \lambda ({\bf s}')$, then there must exist $j$ such that $\m(s_j) > \m'(s_j)$ and $\m(s_i) = \m'(s_i)$ for $i< j$.

Now consider a vector $v = \sum_{\bf s} D_{\bf s}^- p_{\bf s}(d) \otimes 1 \in M_\eta$, and a fixed forest ${\bf s}_0 =(S_0, \m_0)$, where $S_0 =\{s_1, \ldots, s_{\iota ({\bf s}_0)}\}$.  Define 
\begin{align*}
\Lambda(v) &=\{ t\in \T \mid \rootdeg (t) \ge R_\eta, \ \eta (D_t^+) \neq 0, \  \mbox{$|t| > B_\eta + |{{\bf s}}|$ for all $p_{{\bf s}}(d) \neq 0$} \}\\
\Lambda_1^{{\bf s}_0} (v) &= \{ t \in \Lambda (v) \mid  \m_{\C (t)}(s_1) \geq  \m_{\C (t')}(s_1) \ \mbox{for all} \ t' \in \Lambda (v)\},\\ 
\Lambda_i^{{\bf s}_0} (v) &= \{ t \in \Lambda_{i-1}^{{\bf s}_0} (v) \mid \m_{\C (t)}(s_i) \geq  \m_{\C (t')}(s_i) \ \mbox{for} \ t' \in \Lambda_{i-1}^{{\bf s}_0} (v) \},  \  1<i \le \iota( {\bf s}_0),
\end{align*}
where $\m_{\C (t)}$ is the multiplicity function associated with the forest $\C (t)$.  Note that the condition $\rootdeg (t) \ge R_\eta$ in the definition of $\Lambda(v)$ is equivalent to the condition $\rootdeg (t) = R_\eta$, and $\Lambda (v) \neq \emptyset$ since $\eta$ has infinite support.

\begin{lem} \label{lem:actionOrderRootBnd}
Assume $\eta$ is root-bounded with infinite support.  Let  ${\bf s}_0, {\bf s}$ be forests such that ${\bf s}_0 \neq {\bf s}$, $l ({\bf s}_0) = l ({\bf s})$, $|{\bf s}_0| \ge |{\bf s}|$, and $\lambda ( {\bf s}_0) \ge \lambda ({\bf s})$; and let $p(d) \in \cc[d]$.
Then for $v=D_{\bf s}^+ p(d) \otimes 1$ and $t \in \Lambda^{{\bf s}_0}_{\iota ({\bf s}_0)} (v)$,
$$[D_{t \, \cup_{\rt (t)} {\bf s}_0}^+, D_{{\bf s}}^- p(d)] \otimes 1 = 0.$$
\end{lem}
\begin{proof}
Write ${\bf s}_0=(S_0, \m_0)$, $ {\bf s}=(S, \m)$, and $S_{0}=\{s_1, \ldots, s_{\iota ({\bf s}_0)}\}$, with $\m_{0}( s_1) \ge \cdots \ge \m_{0}(s_{\iota({\bf s}_0)})$ as in (\ref{eqn:forestSetMult}).  Because $\lambda ({\bf s}_0) \ge \lambda ({\bf s})$, there exists $j$ such that $\m_{0}(s_i)= \m(s_i)$ for $i<j$ and $\m_{0}(s_j)>\m(s_j)$.  Thus by Lemma \ref{lem:partReplacement}, either $[D_{t \, \cup_{\rt (t)} {\bf s}_0}^+, D_{{\bf s}}^-p(d)] \otimes 1 = 0$ or $[D_{t \, \cup_{\rt (t)} {\bf s}_0}^+, D_{{\bf s}}^-p(d)] \otimes 1 = \xi \eta (D_u^+) q(d) \otimes 1$, where $0 \neq \xi \in \z$, $q(d) \in \cc [d]$, and $u \in \T$ has the property that 
$$
\m_{\C (u)}(s_i) =  \m_{\C (t)}(s_i) \ \mbox{for $i<j$} \ \mbox{and} \ \m_{\C (u)}(s_j) > \m_{\C (t)}(s_j)
$$
for some $j \in \z$.
Note that $|u| = |t| + |{\bf s}_0| - |{\bf s}| > B_\eta + |{\bf s}| + |{\bf s}_0| - |{\bf s}| = B_\eta + |{\bf s}_0| \ge B_\eta + |{\bf s}|$ and $\rootdeg (u) \ge \rootdeg (t)$.  If $\eta (D_u^+) \neq 0$, then $u \in \Lambda_{j-1}^{{\bf s}_0}(v)$ and thus $\m_{\C (u)} (s_j) > \m_{\C (u)} (s_j)$ contradicts $t \in \Lambda_{j}^{{\bf s}_0}(v)$.  Therefore it must be that $\eta (D_u^+) = 0$, and consequently $[D_{t \, \cup_{\rt (t)} {\bf s}_0}^+, D_{{\bf s}}^-p(d)] \otimes 1 = \xi \eta (D_u^+) q(d) \otimes 1 = 0$.
\end{proof}

\begin{thm}\label{thm:rootBoundedSimp}
Let $\eta$ be a non-zero, root-bounded homomorphism.  Then the space of Whittaker vectors of type $\eta$ in $M_\eta$ is one dimensional, and $M_\eta$ is simple.
\end{thm}
\begin{proof}
Given Corollary \ref{cor:finitesupport}, it is enough to prove the result in the case that $\eta$ has infinite support.  

Let $v \in M_\eta \setminus \cc 1 \otimes 1$; we show that $v$ is not a Whittaker vector. Write $v = \sum_{{\bf s}}  D_{\bf s}^- p_{\bf s}(d)\otimes 1$, and note that by Lemma \ref{lem:D+p(d)v} we may assume that there exists at least one nonempty forest ${\bf s}$ with $p_{\bf s}(d) \neq 0$.

Let $k \in \z_{>0}$ be maximal such that $k=l({\bf s})$ for some ${\bf s}$ with $p_{\bf s}(d) \neq 0$, and let $M = \max \{ |{\bf s}| \mid p_{\bf s}(d) \neq 0, \, l( {\bf s}) = k \}$.  Now choose ${\bf s}_0$ such that $\lambda({\bf s_0})$ is maximal (under the lexicographic ordering) among all forests in 
$$\{ {\bf s} \mid p_{\bf s}(d) \neq 0, \, l({\bf s}) = k, \, |{\bf s}| = M \}.$$

Let $t \in \Lambda_{\iota({\bf s_0})}^{\bf s_0}(v)$.  We now have 
\begin{align}
\left( D_{t \, \cup_{\rt (t)} {\bf s}_0}^+ - \eta (D_{t \, \cup_{\rt (t)} {\bf s}_0}^+) \right) v &= \sum_{{\bf s}} [D_{t \, \cup_{\rt (t)} {\bf s}_0}^+, D_{{\bf s}}^-p_{\bf s}(d)] \otimes 1 \nonumber \\ 
&=  [D_{t \, \cup_{\rt (t)} {\bf s}_0}^+, D_{{\bf s}_0}^-p_{\bf s_0}(d)] \otimes 1 \nonumber \\
&\ \ + \sum_{l({\bf s})=l({\bf s_0}), |{\bf s}_0| \ge |{\bf s}|,  {{\bf s}} \neq {\bf s}_0} [D_{t \, \cup_{\rt (t)} {\bf s}_0}^+, D_{{\bf s}}^-p_{\bf s}(d)] \otimes 1 \label{Sum1}\\
&\ \ + \sum_{l({\bf s})<l({\bf s_0})} [D_{t \, \cup_{\rt (t)} {\bf s}_0}^+, D_{{\bf s}}^-p_{\bf s}(d)] \otimes 1 \label{Sum2} \\
&=  [D_{t \, \cup_{\rt (t)} {\bf s}_0}^+, D_{{\bf s}_0}^-p_{\bf s_0}(d)] \otimes 1, \nonumber
\end{align}
where (\ref{Sum1}) is zero by Lemma \ref{lem:actionOrderRootBnd} and (\ref{Sum2}) is zero by Lemma \ref{lem:bigRootDegZero} since $\rootdeg (t) \ge R_\eta$.  Lemma \ref{lem:partReplacement} now implies
\begin{align*}
\left( D_{t \, \cup_{\rt (t)} {\bf s}_0}^+ - \eta (D_{t \, \cup_{\rt (t)} {\bf s}_0}^+) \right) v &=  [D_{t \, \cup_{\rt (t)} {\bf s}_0}^+, D_{{\bf s}_0}^-p_{\bf s_0}(d) ] \otimes 1 \\
&= \xi \eta (D_{t}^+) (p_{\bf s_0}(d-|t|)-p(d))\otimes 1 \quad (0 \neq \xi \in \cc ) \\
& \neq 0.
\end{align*}
Therefore $v$ is not a Whittaker vector, and so the Whittaker vectors of $M_\eta$ are precisely the elements of $\cc 1 \otimes 1$ as desired.
\end{proof}

\subsection{The zero homomorphism}
We also briefly consider the zero homomorphism: ${\bf 0}: \n^+ \rightarrow \cc$ where ${\bf 0}(x) =0$ for all $x \in \n^+$.  Lemma \ref{lem:D+p(d)v} implies that, for any $0 \neq p(d) \in \cc[d]$, the vector $p(d) \otimes 1$ generates a proper submodule of $M_{\bf 0}$.   In particular,  if $p(d) = d- \xi$ for some $\xi \in \cc$, then $M_{\bf 0} / (U(\g) p(d) \otimes 1)$ is isomorphic to the Verma module $M(\xi)$ of highest weight $\xi$. Thus, the simple highest weight modules $L(\xi)$ are quotients of $M_{\bf 0}$ as well. (Though it is shown in Theorem 3.2 of \cite{S} that $M(\xi)$ is usually simple, it is not generally known for which $\xi$ the module $M(\xi)$ is simple.) Moreover, it follows from the arguments found in \cite{OW2013} that any simple object in $\W_{\bf 0}$ (see notation in Section \ref{sec:Whittakerbackground}) is isomorphic to $L(\xi)$.

\section{Connection to Verma modules and central characters}\label{sec:Vermas}

In this section, we discuss connections between Whittaker modules and central characters.  Corollary \ref{cor:cenCharInsElim} uses the existence of a simple standard Whittaker module for the insertion-elimination Lie algebra to imply that all Verma modules admit the same central character.

Since it poses no additional difficulties, we begin by working in a more general context than that of the insertion-elimination Lie algebra.  Let $\g = \n^- \oplus \h \oplus \n^+$ denote a Lie algebra with a triangular decomposition as in \cite{MP95}, and let $Z(\g)$ denote the center of the universal enveloping algebra $U(\g)$.   For $\mu \in \h^*$, define (as in Section 4.2 of \cite{BatMaz} and Section 2.3 of \cite{MP95}) the highest weight Verma module $M(\mu)$ by 
$$M(\mu) := U(\g) \otimes_{U(\h \oplus \n^+)} \cc_\mu,$$ 
where $\cc_\mu$ is the 1-dimensional $\h \oplus \n^+$-module given by $(h + n).1 = \mu(h)$ for $h \in \h$ and $n \in \n^+$.  The lowest-weight Verma module $N(\mu) = U(\g) \otimes_{U(\n^- \oplus \h)} \cc_\mu$ is defined similarly.   We continue to let $M_\eta$ denote the standard Whittaker module corresponding to $\eta : \n^+ \to \cc$.

Let $Z( \g )$ denote the center of the universal enveloping algebra $U(\g)$.  We say that a $\g$-module $V$ admits a central character $\chi : Z(\g) \to \cc$ if we have $zv = \chi (z) v$ for every $z \in Z(\g)$ and $v \in V$.  For $\mu \in \h^*$, it is a standard result that the Verma module $M(\mu)$ admits a central character, and we let 
$$\chi_\mu : Z(\g) \to \cc$$ 
denote the central character corresponding to $M(\mu)$.  

\begin{lem}
Let $\g = \n^- \oplus \h \oplus \n^+$ be a Lie algebra with a triangular decomposition.  For $\mu \in \h^*$, the module $N(-\mu)^*$ has central character $\chi_\mu$.
\end{lem}
\begin{proof}
Since $N(-\mu)$ is a lowest weight module, it must admit a central character, and therefore $N(-\mu)^*$ must also admit a central character $\chi$.   Let $f$ be the unique element of $N(-\mu)^*$ sending $1 \otimes 1$ to $1$ and sending every other weight space to 0.  It is clear that $f$ is a highest weight vector of weight $\mu$.  Therefore $U(\g) f \subseteq N(-\mu)^*$ is a highest weight module of highest weight $\mu$ and must admit the central character $\chi_\mu$.  But $U(\g) f$ is a submodule of $N(-\mu)^*$, which has a central character $\chi$, so it follows that $\chi = \chi_\mu$.
\end{proof}

\begin{prop}\label{prop:cenCharWhiVerm}
Let $\g = \n^- \oplus \h \oplus \n^+$ be a Lie algebra with a triangular decomposition.  Let $\mu \in \h^*$, and let $\chi_\mu$ be the central character corresponding to the Verma module $M(\mu)$.  If $\eta : \n^+ \to \cc$ is such that $M_\eta$ is simple, then $M_\eta$ admits the central character $\chi_\mu$
\end{prop}
\begin{proof}
By a generalization of Schur's Lemma (Cf. \cite[Ex.~2.12.28]{rowen:rt88}), the simple module $M_\eta$ must admit a central character $\xi_\eta : Z(\g) \to \cc$.  For $\mu \in \h^*$, the module $N(-\mu)^*$ has a one-dimensional space of Whittaker vectors of type $\eta$ (see Lemma 37 of \cite{BatMaz}), and thus there is a $\g$-module homomorphism 
$$\varphi : M_\eta \to N(-\mu)^*.$$
As $M_\eta$ is simple, $\varphi$ is necessarily injective.   Now the image of $\varphi$ is a submodule of $N(-\mu)^*$, which admits the central character $\chi_\mu$, so $Z(\g)$ must act by $\chi_\mu$ on the image of $\varphi$.  Then for all $v \in M_\eta$ and $z \in Z(\g)$, we have 
$$\xi_\eta (z) \varphi (v) = \varphi ( \xi_\eta (z) v)) = \varphi ( zv) = z \varphi (v)  = \chi_\mu (z) \varphi (v),$$
and this shows that $\xi_\eta = \chi_\mu$ as long as $M_\eta$ is simple.  
\end{proof}

\begin{cor}
Assume $\g = \n^- \oplus \h \oplus \n^+$ has a  triangular decomposition.  If there exists $\eta : \n^+ \to \cc$ such that the standard Whittaker module $M_\eta$ is simple, then every Verma module admits the same central character.  
\end{cor}
\begin{proof}
If $\xi_\eta$ is the central character admitted by $M_\eta$, then for $\mu_1, \mu_2 \in \h^*$, Proposition \ref{prop:cenCharWhiVerm} implies that $\chi_{\mu_1} = \xi_\eta = \chi_{\mu_2}$, as desired.
\end{proof}

\begin{cor}
Let $\g = \n^- \oplus \h \oplus \n^+$ be a Lie algebra with a triangular decomposition.  Then any  two simple standard Whittaker modules admit the same central character.
\end{cor}
\begin{proof}
If $M_{\eta_1}$ and $M_{\eta_2}$ are simple, for $\eta_1 : \n^+ \to \cc$ and $\eta_2 : \n^+ \to \cc$, then for $\mu \in \h^*$, Proposition \ref{prop:cenCharWhiVerm} implies both $M_{\eta_1}$ and $M_{\eta_2}$ have the central character $\chi_\mu$.  
\end{proof}

Since the standard Whittaker module $M_\eta$ for the insertion-elimination Lie algebra is simple whenever $\eta$ has finite support or is root bounded, we conclude the following, which suggests that $Z(\g) = \cc 1$ for the insertion-elimination Lie algebra.

\begin{cor}\label{cor:cenCharInsElim}
Every Verma module for the insertion-elimination Lie algebra admits the same central character.  
\end{cor}

\end{document}